\documentclass{amsart}

\usepackage{amsmath}
\usepackage{amscd}
\usepackage{amsfonts}
\usepackage{amssymb}

\theoremstyle{theorem}
\newtheorem{theorem}{Theorem}[section]
\newtheorem{lemma}[theorem]{Lemma}

\theoremstyle{definition}

\newtheorem{remark}[theorem]{Remark}

\newcommand{\be}{\begin{equation}}
\newcommand{\ee}{\end{equation}}
\newcommand{\bes}{\begin{equation*}}
\newcommand{\ees}{\end{equation*}}

\newcommand{\cH}{\mathcal{H}}
\newcommand{\cK}{\mathcal{K}}

\newcommand{\cB}{\mathcal{B}}

\newcommand{\cA}{\mathcal{A}}

\newcommand{\ad}{\operatorname{ad}}
\newcommand{\alg}{\operatorname{Alg}}

\begin{document}

\allowdisplaybreaks

\title{Unitary dilation of freely independent contractions}
\author{Scott Atkinson and Christopher Ramsey}
\address{University of Virginia, Charlottesville, VA, USA}
\email{saa6uy@virginia.edu and cir6d@virginia.edu}

\begin{abstract}
Inspired by the Sz.-Nagy-Foias dilation theorem we show that  $n$ freely independent contractions dilate to $n$ freely independent unitaries. 
\end{abstract}


\maketitle

%%%%%%%%%%%%%%%%%%%%
%%%%%%%%%%%%%%%%%%%%
\section{Introduction} 

The Sz.-Nagy-Foias dilation theorem is a celebrated result in classical dilation theory.  It says that $n$ doubly commuting contractions  can be simultaneously dilated to $n$ doubly commuting unitaries. This was the original multivariable dilation theory context proven by Brehmer and Sz.-Nagy \cite{Brehmer, Sz.-Nagy, Sz.-NagyFoias} until And\^o \cite{Ando} proved that one can do this for just commuting and not doubly commuting contractions when $n=2$. However, it was subsequently shown in \cite{Parrott} and \cite{Varopoulos} that there are three commuting contractions which do not dilate to three commuting unitaries. This obstruction spurred on dilation theories in other contexts \cite{Arveson, Bunce, Drury, Frazho, Popescu} and many other generalizations. One recent usage of dilations of doubly commuting contractions is the dilation of Nica covariant representations of lattice-ordered semigroups \cite{Fuller, Li}.

Doubly commuting is one of two ingredients in the notion of tensor independence (or classical independence). It is natural then to ask whether $n$ tensor independent contractions can be dilated to $n$ tensor independent unitaries. The answer is yes (Theorem \ref{tensorind}) and begs the question whether this can be done with other notions of non-commutative probability, namely free probability. 

Stemming from the notion of reduced free product \cite{Avitzour, Voiculescu} Voiculescu developed the theory of free probability in the 1980's with the goal of solving the free group factor problem. While this still remains unsolved, free probability has become a very important field of mathematical research. For further reading see \cite{HiaiPetz, NicaSpeicher}.

This paper culminates in Theorem \ref{freedilation}, that $n$ freely independent contractions do indeed dilate to $n$ freely independent unitaries. In a dilation theory context this has been done by Boca in \cite{Boca} where he gives the most general unitary dilation of $n$ contractions. The only free probability dilation result we know of is the unitary dilation of L-free sets of contractions of Popa and Vaes \cite{PopaVaes}. 

\subsection*{Acknowledgements} We would like to thank  David Sherman and Stuart White for some very helpful discussions during White's visit to the University of Virginia sponsored by the Institute of Mathematical Science.

%%%%%%%%%%%%%%%%%%%%
%%%%%%%%%%%%%%%%%%%%
\section{Dilation theory of tensor independence}

We first turn to the classical setting for inspiration. There are many great proofs of the Sz.-Nagy-Foias dilation theorem and the following constructive method is probably quite old but the authors have only seen it written down in \cite[Example 2.5.13]{DFK}. Recall that two operators $S,T\in B(\cH)$ doubly commute if $ST = TS$ and $S^*T= TS^*$. This is equivalent to requiring that $C^*(1,S)$ and $C^*(1,T)$ commute. Note that this does not require that $S$ and $T$ are normal.

\begin{theorem}[Sz.-Nagy-Foias]\label{sznagyfoias}
Let $T_1,\dots, T_n$ be doubly commuting contractions in $B(\cH)$.  Then there exists a Hilbert space $\cK$ containing $\cH$ and doubly commuting unitaries $U_1,\dots,U_n \in B(\cK)$ such that 
\[
T_1(k_1)\cdots T_n(k_n) = P_\cH U_1^{k_1}\cdots U_n^{k_n}|_\cH, \ \ \textrm{where} \ \ T(k) = \left\{\begin{array}{ll} T^k, & k\geq 0 \\ T^{*-k}, & k < 0\end{array}\right. .
\]
Furthermore, this dilation is unique up to unitary equivalence when $\cK$ is minimal, meaning that it is the smallest reducing subspace of $U_1,\dots, U_n$ containing $\cH$.
\end{theorem}

\begin{proof}
As mentioned above, this can also be found in \cite[Example 2.5.13]{DFK}. Let $\cH_1 = \ell_2(\mathbb Z) \otimes \cH$ and set
\[
T_1^{(1)} = \left[\begin{array}{ccccccc} \ddots&\\ \ddots&0 \\ &I & 0 \\ &&D_{T_1^*} & T_1 \\ &&-T_1^* & D_{T_1} & 0 \\ &&&&I&0 \\ &&&&&\ddots&\ddots\end{array}\right] \ \ \textrm{and} \ \ T_j^{(1)} = I_{\ell_2(\mathbb Z)} \otimes T_j, \ 2\leq j\leq n.
\]
Note that $T_1^{(1)}$ is the classic Sch{\"a}ffer form of the Sz.-Nagy dilation of $T_1$ and so is a unitary \cite{Schaffer}. Since $T_1,\dots, T_n$ doubly commute then so do $T_1^{(1)}, \dots, T_{n}^{(1)}$, this is immediate after noticing that the defect operators $D_{T_1} = (I - T_1^*T_1)^{1/2}$ and $D_{T_1^*}$ are in $C^*(1,T_1)$.

In the second step, let $\cH_2 = \ell_2(\mathbb Z) \otimes \cH_1$, $T_2^{(2)}$ be the Sch{\"a}ffer-Sz.-Nagy dilation of $T_2^{(1)}$ and $T_j^{(2)} = I_{\ell_2(\mathbb Z)} \otimes T_j^{(1)}$ for $j\neq 2$. Then $T_1^{(2)}$ and $T_2^{(2)}$ are unitaries and $T_1^{(2)},\dots, T_n^{(2)}$ are doubly commuting.

Continuing in this way one arrives at the nth step with doubly commuting unitaries $T_1^{(n)},\dots, T_n^{(n)}\in B(\cH_n)$ that are easily seen to satisfy the joint power dilation condition. 
\vskip 6 pt

Uniqueness when the dilation is minimal follows in the same way as in the one variable setting. It is proven by way of the uniqueness of the minimal Stinespring representation. 
\end{proof}

This can be rephrased into a non-commutative probability context. Recall that a non-commutative C$^*$-probability space $(\cA, \varphi)$ is a C$^*$-algebra $\cA$ along with a state $\varphi\in S(\cA)$.
We say that the operators  $T_1,\dots, T_n \in \cA$ are tensor (or classically) independent in $(\cA, \varphi)$ (or with respect to $\varphi$) if $C^*(1,T_1),\dots, C^*(1,T_n)$ pairwise commute and given $a_i\in C^*(1,T_i)$ we have the following factorization
\[
\varphi\left(\prod_{i=1}^n a_i\right) = \prod_{i=1}^n \varphi(a_i).
\]

%%%%%%%%
\begin{theorem}\label{tensorind}
Let $T_1,\dots, T_n \in B(\cH)$ be tensor independent contractions in the non-commutative probability space $(B(\cH),\varphi)$. Then there exists a Hilbert space $\cK$ containing $\cH$ and unitaries $U_1,\dots, U_n\in B(\cK)$ that are tensor independent with respect to the state $\psi = \varphi\circ \ad_{P_{\cH}}$ such that 
\[
T_1(k_1)\cdots T_n(k_n) = P_\cH U_1^{k_1}\cdots U_n^{k_n}|_\cH, \ \ \textrm{where} \ \ T(k) = \left\{\begin{array}{ll} T^k, & k\geq 0 \\ T^{*-k}, & k < 0\end{array}\right. .
\]
Furthermore, this dilation is unique up to unitary equivalence when $\cK$ is minimal, meaning that it is the smallest reducing subspace of $U_1,\dots, U_n$.
\end{theorem}

\begin{proof}
All that needs to be shown is that the unitaries arising from Theorem \ref{sznagyfoias} are tensor independent with respect to $\psi$. To this end first assume that $\cK = \cH_n$ and $U_j  = T_j^{(n)}$ as in the proof of Theorem \ref{sznagyfoias}. 

Let $a_i \in C^*(1, T_i^{(1)}), 1\leq i\leq n$.  This implies that $P_\cH a_i|_\cH \in C^*(1,T_i)$ for $1\leq i\leq n$, and $P_\cH$ and $a_i$ commute for $2\leq i\leq n$ since $\cH$ is a reducing subspace for each $C^*(1,T_i^{(1)}), 2 \leq i \leq n$. Hence, 
\begin{align*}
\varphi\left(P_\cH\prod_{i=1}^n a_i P_\cH\right) &= \varphi\left(\prod_{i=1}^n(P_\cH a_iP_\cH)\right) 
\\ &= \prod_{i=1}^n \varphi(P_\cH a_i P_\cH).
\end{align*}
Thus, $T_1^{(1)}, \dots, T_n^{(1)}$ are tensor independent with respect to $\varphi\circ\ad_{P_\cH}$. Continuing in this fashion one gets that $T_1^{(n)}, \dots, T_n^{(n)}$ are tensor independent with respect to $\varphi\circ\ad_{P_\cH} \circ \ad_{P_{\cH_1}} \circ \cdots \circ \ad_{P_{\cH_{n-1}}} = \varphi\circ \ad_{P_\cH} = \psi$ where the last copy of $\cH$ is in $\cH_n$.

Uniqueness of this dilation is given by Theorem \ref{sznagyfoias}.  %Probably add a little here
\end{proof}

%%%%%%%%%%%%%%%%%%%%%%%
%%%%%%%%%%%%%%%%%%%%%%%
\section{Dilation theory of free independence}

In this section we will prove a theorem very similar to Theorem \ref{tensorind} in another non-commutative probability context. Recall that the operators $T_1,\dots, T_n \in \cA$ are freely independent (or $*$-free) in $(\cA, \varphi)$ if their C$^*$-algebras $C^*(1,T_1), \dots, C^*(1,T_n)$ are freely independent. That is, whenever $a_{j} \in C^*(1,T_{i_j})$ such that $\varphi(a_j) = 0$ for $1\leq i_j\leq n$ and $i_j \neq i_{j-1}$ for $1 < j\leq m$ then 
\[
\varphi(a_1a_2\cdots a_m) = 0.
\]

Another proof of the Sz.-Nagy-Foias Theorem (Theorem \ref{sznagyfoias} above) can be found in Paulsen \cite[Theorem 12.10]{Paulsen}. Here one gets ucp maps $\theta_i : C(\mathbb T) \rightarrow C^*(1,T_i)$ given by dilation theory. Now one can extend this to the ucp map $\theta_1\otimes\cdots\otimes \theta_n$ on $C(\mathbb T)\otimes_{\rm max} \cdots \otimes_{\rm max} C(\mathbb T) \simeq C(\mathbb T^n)$. By taking the Stinespring representation of $\theta$ one gets the desired doubly commuting unitaries that jointly dilate $T_1,\dots, T_n$.

This provides a roadmap for an attempt to prove the free analogue of Theorem \ref{tensorind}. Namely, by using the free product of ucp maps  and then taking the Stinespring representation of this map it will be shown that one gets unitaries that jointly dilate $T_1,\dots, T_n$. One then hopes that these unitaries will be $*$-free with respect to a natural state. 

This approach, minus the free probability is exactly what Boca \cite{Boca} uses to establish his unitary dilation result, in fact proving a more general statement about normal rational dilations.

%%%%%

Recall now, that the unital  universal free product of unital C$^*$-algebras $\cA_1,\dots, \cA_n$ is the universal C$^*$-algebra amalgamated over $\mathbb C$ generated by $\cA_1,\dots, \cA_n$ and is denoted ${\check *}_{i=1}^n \cA_i$. 
In particular, whenever one has a unital C$^*$-algebra $\cB$ and unital $*$-homomorphisms $\pi_i:\cA_i \rightarrow \cB$ 
then there exists a unital $*$-homomorphism $\pi:  {\check *}_{i=1}^n \cA_i \rightarrow \cB$.

Suppose there are unital completely positive maps $\theta_i : \cA_i \rightarrow \cB$ and states $\varphi_i\in S(\cA_i)$.
In \cite{Boca}, Boca proves that there exists a ucp map $*_{i=1}^n \theta_i = \theta_1 * \cdots *\theta_n : {\check *}_{i=1}^n \cA_i \rightarrow \cB$ such that $*_{i=1}^n \theta|_{\cA_i} = \theta_i$. This is defined on {\it reduced words} with respect to the $\varphi_i$. Namely, when $a_j \in \cA_{i_j}$ with $\varphi_{i_j}(a_j)=0$ and $i_j\neq i_{j-1}$ then
\[
*_{i=1}^n \theta(a_1\cdots a_m) = \theta_{i_1}(a_1)\cdots \theta_{i_m}(a_m).
\]
This completely determines $*_{i=1}^n \theta$ as (reduced words + $\mathbb C 1$) is dense in ${\check *}_{i=1}^n \cA_i$.

%%%%%
\begin{lemma}\label{thetaisahom}
Suppose $T_1,\dots, T_n\in B(\cH)$ and $V_1,\dots, V_n\in B(\cK)$ are contractions and $\theta_i : C^*(1,V_i) \rightarrow C^*(1,T_i)$ such that $p(V_i)  \mapsto p(T_i)$ are ucp maps ($p$ a polynomial). If $\psi_i\in S(C^*(1,V_i))$ then the free product ucp map with respect to the $\psi_i$, $*_{i=1}^n \theta_i$, is a homomorphism on the subalgebra $\overline\alg\{1,V_1,\dots, V_n\}$ of ${\check *}_{i=1}^n C^*(V_i)$.
\end{lemma}

\begin{proof}
The result can be established by induction.
By definition each $\theta_i$ is already a homomorphism on $\overline{\alg}\{1,V_i\} \subseteq C^*(V_i)$.
Now for $m\geq 1$ assume that for all $1\leq k\leq m$ and for any $b_j \in \overline{\alg}\{1,V_{i_j}\}, 1\leq j\leq k$ we have that $*_{i=1}^n \theta_i\left(\prod_{j=1}^k b_j\right) = \prod_{j=1}^k \theta_{i_j}(b_j)$.

Suppose now we have $a_j\in \overline{\alg}\{1,V_{i_j}\}, 1\leq j\leq m+ 1$. If a pair of neighboring terms belongs to the same algebra, say $i_j = i_{j-1}$, by the inductive hypothesis and since $\theta_{i_j}$ is a homomorphism we have
\begin{align*}
*_{i=1}^n \theta_i(a_1\cdots a_{j-1}a_j \cdots a_{m+1}) &= \theta_{i_1}(a_1)\cdots \theta_{i_j}(a_{j-1}a_j)\cdots \theta_{i_{m+1}}(a_{m+1})
\\ &= \theta_{i_1}(a_1)\cdots \theta_{i_{j-1}}(a_{j-1})\theta_{i_j}(a_j)\cdots \theta_{i_{m+1}}(a_{m+1}).
\end{align*}
Otherwise assume that $i_{j-1} \neq i_j$ for $1< j\leq m+1$ and then we have
\begin{align*}
*_{i=1}^n \theta_i\left(\prod_{j=1}^{m+1} a_j\right) =& \ *_{i=1}^n \theta_i\left( \prod_{j=1}^{m+1} (a_j - \psi(a_j))\right) +
\\ & \ *_{i=1}^n \theta_i\left(\psi(a_1)\prod_{j=2}^{m+1} (a_j - \psi(a_j))\right) + 
\\ & \ *_{i=1}^n \theta_i\left(a_1\psi(a_2)\prod_{j=3}^{m+1} (a_j - \psi(a_j))\right) + \cdots +
\\ & \ *_{i=1}^n \theta_i\left(a_1\cdots a_m \psi(a_{m+1})\right)
\\ = & \ \prod_{j=1}^{m+1}\theta_{i_j}(a_j- \psi(a_j)) \ + 
\\ & \ \psi(a_1)\prod_{j=2}^{m+1}\theta_{i_j}(a_j - \psi(a_j)) \ + 
\\ & \ \theta_{i_1}(a_1)\psi(a_2)\prod_{j=3}^{m+1}\theta_{i_j}(a_j - \psi(a_j))\ + \cdots +
\\ & \ \theta_{i_1}(a_1)\cdots \theta_{i_m}(a_m)\psi(a_{m+1})
\\ = & \prod_{j=1}^{m+1} \theta_{i_j}(a_j)
\end{align*}
by the definition of $*_{i=1}^n \theta_i$ and the inductive hypothesis.
 
 The result follows as it is a simple matter now to show that $*_{i=1}^n \theta_i(ab) = *_{i=1}^n \theta_i(a)*_{i=1}^n \theta_i(b)$ for all $a,b\in \overline\alg\{1,V_1,\dots, V_n\}$.
\end{proof}

%%%%%%%%%%
\begin{theorem}\label{freedilation}
Let $T_1,\dots, T_n\in B(\cH)$ be freely independent contractions in the non-commutative probability space $(B(\cH), \varphi)$. Then there exists a Hilbert space $\cK$ containing $\cH$ and unitaries $U_1,\dots, U_n\in B(\cK)$ that are freely independent with respect to $\varphi\circ\ad P_\cH$ such that
\[
T_{i_1}^{k_1}\cdots T_{i_m}^{k_n} = P_\cH U_{i_1}^{k_1}\cdots U_{i_m}^{k_n}|_\cH, \ \ 1\leq i_j\leq n \ \textrm{and} \   k_j\in \mathbb N\cup \{0\}.
\]
Furthermore, this dilation is unique up to unitary equivalence when $\cK$ is minimal.
\end{theorem}

\begin{proof}
For each $1\leq i\leq n$, let $V_i\in B(\cK_i)$ with $\cH\subset \cK_i$ be the minimal unitary dilation of $T_i$ and $\theta_i: C^*(V_i) \rightarrow C^*(1,T_i)$ given by $\theta_i = \ad P_\cH$, a unital completely positive map. Thus, $\psi_i := \varphi\circ\theta_i$ is a state on $C^*(V_i)$ such that $\psi_i(V_i^n) = \varphi(T_i^n), \forall n\geq 0$.

Consider now, the free product ucp map $*_{i=1}^n \theta_i : \check{*}_{i=1}^n C^*(V_i) \rightarrow C^*(1,T_1,\dots, T_n)$ relative to the states $\psi_i$.
Let $(\pi, \cK)$ be the Stinespring representation of $*_{i=1}^n \theta_i$ with $\cH\subset \cK$ which gives $*_{i=1}^n \theta_i(a) = P_\cH\pi(a)|_\cH$. Define unitaries $U_i := \pi(V_i)$ and note that for $m\geq 1$ and $k_i \geq 0$ then
\begin{align*}
P_\cH U_{i_1}^{k_1}\cdots U_{i_m}^{k_m}|_\cH &= P_\cH \pi(V_{i_1}^{k_1}\cdots V_{i_m}^{k_m})|_\cH
\\ &= *_{i=1}^n \theta_i(V_{i_1}^{k_1}\cdots V_{i_m}^{k_m})
\\ &= T_{i_1}^{k_1}\cdots T_{i_m}^{k_m} & \textrm{by Lemma \ref{thetaisahom}.}
\end{align*}
Furthermore, for $a_j\in C^*(U_{i_j})$ with $\varphi\circ\ad P_\cH(a_j) = 0$ and $i_j \neq i_{j-1}$ we have that $\exists b_j\in C^*(V_{i_j})$ such that $\pi(b_j) = a_j$, $\theta_{i_j}(b_j) \in C^*(1,T_{i_j})$ and
\[
0 = \varphi\circ\ad P_\cH(a_j) = \varphi\circ\ad P_\cH\circ \pi(b_j) = \varphi(\theta_{i_j}(b_j)).
\]
Hence,
\begin{align*}
\varphi\circ\ad P_\cH(a_1\cdots a_m) &= \varphi\circ\ad P_\cH\circ\pi(b_1\cdots b_m) \\
&= \varphi\circ *_{i=1}^n \theta_i(b_1\cdots b_m)\\ 
&= \varphi(\theta_{i_1}(b_1)\cdots\theta_{i_m}(b_m))\\
&= 0.
\end{align*}
Therefore, $U_1,\dots, U_n$ are $*$-free with respect to $\varphi\circ\ad P_\cH$.

An argument to show that the minimal unitary dilation is unique is given in \cite[Section 4]{Boca} following from a classic remark of Durszt and Sz.-Nagy \cite{DursztSzNagy}. In particular, given two minimal unitary dilations $U_1,\dots, U_n$ and $U_1',\dots, U_n'$ on $\cK$ and $\cK'$ respectively, there is a unitary $\Theta: \cK \rightarrow \cK'$ fixing the subspace $\cH$ such that $\Theta U_i = U_i'\Theta$. From this it is easy to see that the states $\varphi\circ\ad P_\cH$ are equal by way of $\Theta$ since a
 state is completely determined by the values 
 \[
 \varphi\circ\ad P_\cH(U_i^k) = \varphi(T_i^k) = \varphi\circ\ad P_\cH((U_i)')^k), \ \forall k\in\mathbb N
 \] because $U_1,\dots, U_n$ and $U_1',\dots, U_n'$ are collections of $*$-free unitaries \cite[Lemma 5.13]{NicaSpeicher}. 
\end{proof}

\begin{remark}
It should be noted, by \cite[Proposition 2.5.3]{VoicDykeNica}, $\varphi\circ\ad P_\cH$ is a tracial state on $C^*(U_1,\dots, U_n)$ since $U_1,\dots, U_n$ are $*$-free and $\varphi\circ\ad P_\cH$ is a trace on each $C^*(U_i)$ (the algebra is commutative).
\end{remark}

One of the purposes of doing the previous arguments carefully is that we can now say what happens when the state $\varphi$ is faithful.

%%%%%
\begin{lemma}\label{faithfuldilation}
Let $T\in B(\cH)$ be a contraction and $\varphi$ be a faithful state on $C^*(1,T)$. If $U\in B(\cK)$ with $\cH \subset \cK$ is the minimal unitary dilation of $T$ then $\varphi\circ\ad P_\cH$ is a faithful state on $C^*(U)$.
\end{lemma}

\begin{proof}
Let $(\pi, \cK', \xi)$ be the GNS representation of $(C^*(U), \varphi\circ\ad P_\cH)$. Then $\pi(U)$ is still a unitary, $\varphi\circ \ad P_\cH(a) = \langle \pi(a)\xi,\xi\rangle$ and $\langle \cdot \xi,\xi\rangle$ is a faithful state on $\pi(C^*(U)) = C^*(\pi(U))$.

Suppose $a$ is a positive element in $\ker \pi$. Then $0 = \langle \pi(a)\xi,\xi\rangle = \varphi\circ\ad P_\cH(a)$ which implies that $\ad P_\cH(a) = 0$ since compression to $\cH$ is a completely positive map and $\varphi$ is faithful. Since $\ker \pi$ is a $C^*$-algebra (so every element of $\ker \pi$ can be written as a linear combination of at most four positive elements from $\ker \pi$), we can conclude that $\ker \pi \subset \ker \ad P_\cH$. Thus there exists a well-defined ucp map $\theta : C^*(\pi(U)) \rightarrow C^*(1,T)$ given by sending $\pi(a) \mapsto P_\cH a|_\cH$. Notably, we have $\theta(\pi(U)^n) = T^n, n\geq 0$.

Let $(\tilde\pi, \cK'')$ with $\cH\subset \cK''$ be the minimal Stinespring representation of $\theta$. That is, $\tilde\pi : C^*(\pi(U)) \rightarrow B(\cK'')$ is a $*$-homomorphism (in fact a $*$-isomorphism) such that $\theta(a) = P_\cH\tilde\pi(a)|_\cH$ and $\cK''$ is the closed linear span of $\tilde\pi(C^*(\pi(U)))\cH$ by minimality.
Define $V := \tilde\pi(\pi(U))$ a unitary and note that $P_\cH V^n|_\cH = \theta(\pi(U^n)) = T^n$. Thus, because of this and  the minimality of the Stinespring representation we have that $V$ is a minimal unitary dilation of $T$. 

Consider now the state $\psi := \varphi\circ\ad P_\cH$ on $C^*(V)$. Now
\begin{align*}
\psi\circ\tilde\pi\left(\sum_{i=-n}^n \alpha_i \pi(U)^n\right) &= \psi\left( \sum_{i=-n}^n \alpha_i V^n\right)
\\ &= \varphi\left(\sum_{i=-n}^n \alpha_i T(n)\right)
\\ &= \varphi\circ\ad P_\cH\left(\sum_{i=-n}^n \alpha_i U^n\right)
\\ &= \left\langle \left(\sum_{i=-n}^n \alpha_i \pi(U)^n\right) \xi, \xi\right\rangle.
\end{align*}
Hence, $\psi\circ\tilde\pi(\cdot) = \langle \cdot \xi,\xi\rangle$ is a faithful state on $C^*(\pi(U))$ and so $\psi$ is a faithful state on $C^*(V)$.

By minimality of the dilations there exists a unitary $W: \cK \rightarrow \cK''$ such that $Wh = h$ for all $h\in \cH$ and $WUW^* = V$. This implies that $\varphi\circ\ad P_\cH$ on $C^*(U)$ is equal to $\psi\circ \ad W$ which is faithful. Therefore, $\varphi\circ\ad P_\cH$ was a faithful state on $C^*(U)$ all along.
\end{proof}

One last ingredient before presenting Theorem \ref{freedilationfaithful} is the reduced free product of C$^*$-probability spaces $(\cA_i, \varphi_i)$, denoted $(\cA, \varphi)$ with $\cA = *_{i=1}^n \cA_i$. As mentioned in the introduction this was introduced by both Avitzour \cite{Avitzour} and Voiculescu \cite{Voiculescu} in the 1980's.

In correspondence with Boca's result on free products of ucp maps there is the reduced free product of ucp maps given by Choda-Blanchard-Dykema \cite{Choda, BlanchardDykema} and the operator-valued conditionally free product of M\l{}otkowski \cite{Mlot} based on the conditionally free product of Bo{\.z}ejko, Leinert and Speicher \cite{BLS}. 

Now we can present the following theorem.

%Faithful version
\begin{theorem}\label{freedilationfaithful}
Let $T_1,\dots, T_n\in B(\cH)$ be freely independent contractions in the non-commutative probability space $(B(\cH), \varphi)$ with $\varphi$ faithful. If $U_1,\dots, U_n\in B(\cK)$ is the minimal free unitary dilation then $\varphi\circ\ad P_\cH$ is faithful.  In particular, when $\varphi$ is faithful, the minimal free unitary dilation of $T_1,\dots, T_n, \varphi$ arises from the reduced free product of their minimal unitary dilations.
\end{theorem}

\begin{proof}
Because of the faithfulness of $\varphi$, $C^*(1,T_1,\dots, T_n)$ is in fact $*$-isomorphic to $*_{i=1}^n (C^*(1,T_i), \varphi)$ \cite[Lemma 1.3]{DykemaRordam}. Let $V_i\in B(\cK_i)$ be the minimal unitary dilation of $T_i$ and again let $\theta_i:C^*(V_i) \rightarrow C^*(1,T_i)$ be the ucp map given by $\theta_i(V_i^k) = T_i^k$. By Lemma \ref{faithfuldilation} $\varphi\circ\theta_i$ is a faithful state.

As mentioned, Choda \cite{Choda} showed that there is a reduced free product of ucp maps, though with a gap in the proof that was filled by Blanchard-Dykema \cite{BlanchardDykema}.
Thus, there exists a ucp map 
\[
\theta : (*_{i=1}^n C^*(V_i), *_{i=1}^n \varphi\circ\theta_i) \rightarrow (*_{i=1}^n C^*(1,T_i), \varphi)
\]
extending each of the $\theta_i$. By \cite[Theorem 7.9]{NicaSpeicher} $*_{i=1}^n \varphi\circ\theta_i = \varphi\circ\theta$ is a faithful state. Lemma \ref{thetaisahom} applies equally well in the reduced free product case to give that $\theta$ acts homomorphically on $\overline{\alg}\{1,V_1,\dots,V_n\}$.

As in Theorem \ref{freedilation}, let $(\pi, \cK)$ be the minimal Stinespring dilation of $\theta$ and $U_i = \pi(V_i)$. Then, $U_1,\dots, U_n, \varphi\circ\ad P_\cH$ is the minimal free unitary dilation of $T_1,\dots, T_n, \varphi$ with $\varphi\circ\ad P_\cH$ faithful.

In fact, we can say a little more. Because $\varphi\circ\theta$ is faithful then $\pi$, the Stinespring representation of $\theta$, has to be faithful, i.e. injective. Hence, again by \cite[Lemma 1.3]{DykemaRordam} $(C^*(U_1,\dots, U_n),\varphi\circ\ad P_\cH) \simeq *_{i=1}^n (C^*(V_i), \varphi\circ\theta_i)$. Therefore, the minimal free unitary dilation of $T_1,\dots, T_n, \varphi$ when $\varphi$ is faithful arises from the reduced free product of their minimal unitary dilations.
\end{proof}

% The bibliography

\end{document}